\documentclass[letterpaper,11pt]{amsart}

\usepackage[margin=1.2in]{geometry}
\usepackage{amsmath,amsthm,amssymb}
\usepackage{xspace,xcolor}
\usepackage[breaklinks,colorlinks,citecolor=teal,linkcolor=teal,urlcolor=teal,pagebackref,hyperindex]{hyperref}
\usepackage[alphabetic]{amsrefs}
\usepackage{color}

\theoremstyle{plain}
\newtheorem{thm}{Theorem}[section]

\newtheorem{lem}[thm]{Lemma}
\newtheorem{prop}[thm]{Proposition}

\theoremstyle{definition}

\theoremstyle{remark}
\newtheorem{rmk}[thm]{Remark}

\def\A{{\mathbf A}}

\def\cD{\mathcal{D}}

\def\cH{\mathcal{H}}
\def\cI{\mathcal{I}}

\def\cO{\mathcal{O}}

\def\.{\cdot}
\def\^{\widehat}

\def\({\left(}
\def\){\right)}

\renewcommand{\and}{ \ \ \text{ and } \ \ }

\usepackage{soul}

\begin{document}

\author[Q.~Chen]{Qianyu Chen}

\address{Department of Mathematics, University of Michigan, 530 Church Street, Ann Arbor, MI 48109, USA}

\email{qyc@umich.edu}

\author[B.~Dirks]{Bradley Dirks}

\address{Department of Mathematics, Stony Brook University, Stony Brook, NY 11794, USA}

\email{bradley.dirks@stonybrook.edu}

\author[M.~Musta\c{t}\u{a}]{Mircea Musta\c{t}\u{a}}

\address{Department of Mathematics, University of Michigan, 530 Church Street, Ann Arbor, MI 48109, USA}

\email{mmustata@umich.edu}

\thanks{Q.C. was partially supported by NSF grant DMS-1952399, B.D. was partly supported by NSF grant DMS-2001132 and NSF-MSPRF grant DMS-2303070,  and M.M. was partially supported by NSF grants DMS-2301463 and DMS-1952399.}


\title[The minimal exponent of cones]{The minimal exponent of cones over smooth complete intersection projective varieties}

\dedicatory{In memory of Lucian B\u{a}descu}

\baselineskip 16pt \footskip = 32pt

\begin{abstract}
We compute the minimal exponent of the affine cone over a complete intersection of smooth projective hypersurfaces intersecting transversely. The upper bound 
for the minimal exponent is proved, more generally, in the weighted homogeneous setting, while the lower bound is deduced from a general lower bound in terms 
of a strong factorizing resolution in the sense of Bravo and Villamayor.
\end{abstract}

\maketitle

\section{Introduction}

Let $X$ be a smooth complex algebraic variety. If $Z$ is a nonempty hypersurface in $X$, then the \emph{minimal exponent} $\widetilde{\alpha}(Z)$ 
was defined by Saito in \cite{Saito-B} using the Bernstein-Sato polynomial of a local equation of $Z$, as follows. Recall that if $Z$ is defined
in an open subset $U$ of $X$ by $f\in\cO_X(U)$, then the Bernstein-Sato polynomial of $f$ is the monic polynomial $b_f(s)\in {\mathbf C}[s]$
of minimal degree such that 
$$b_f(s)f^s\in {\mathcal D}_U[s]\cdot f^{s+1}.$$
Here $f^s$ is a formal symbol on which the sheaf $\cD_U$ of differential operators on $U$ acts in the expected way. 
By a result of Kashiwara \cite{Kashiwara}, all roots of $b_f$ are negative rational numbers.
It is easy to see, by specializing
$s$ to $-1$, that if $Z\vert_U:=Z\cap U$ is nonempty, then $b_f(-1)=0$. By definition, $\widetilde{\alpha}(Z\vert_U)=\widetilde{\alpha}(f)$ is the negative of the largest root of $b_f(s)/(s+1)$
(with the convention that this is $\infty$ if $b_f(s)=s+1$). In order to define $\widetilde{\alpha}(Z)$, one takes an open cover $X=\bigcup_iU_i$
and
$\widetilde{\alpha}(Z)=\min_i\widetilde{\alpha}(Z\vert_{U_i})$, where the minimum is over those $i$ such that $Z\vert_{U_i}$ is nonempty. 

The minimal exponent of a hypersurface is an interesting invariant. 
A result due to Lichtin and Koll\'{a}r \cite{Kollar} says that the minimal exponent refines an important invariant of singularities in birational geometry, the 
\emph{log canonical threshold} ${\rm lct}(X,Z)$; more precisely, we have
$${\rm lct}(X,Z)=\min\big\{\widetilde{\alpha}(Z),1\big\}.$$
It was shown by Saito \cite{Saito-B} that $\widetilde{\alpha}(Z)>1$ if and only if $Z$ has rational singularities. Moreover, we have $\widetilde{\alpha}(Z)=\infty$
if and only if $Z$ is smooth.
Recently, it was shown that the minimal exponent characterizes the \emph{higher Du Bois} property of the singularities of $Z$ (see \cite{MOPW} and \cite{Saito_et_al})
and the condition for \emph{higher rational singularities} (see \cite{FL} and \cite{MP2}).

If $Z$ has isolated singularities, then the minimal exponent can be described via asymptotic expansions of integrals along vanishing cycles,
see \cite{Malgrange} and \cite{Malgrange2}. In this incarnation,  it has been extensively studied in
\cite{AGZV} and is also known as the \emph{Arnold exponent} of $f$.

In \cite{CDMO}, the authors of the present article and Sebasti\'{a}n Olano introduced and studied an extension of the minimal exponent $\widetilde{\alpha}(Z)$ to the case when
$Z$ is a complete intersection in $X$ of pure codimension $r$, for any $r\geq 1$. The definition was in terms of the Kashiwara-Malgrange filtration associated
to $Z$ (the corresponding description in the hypersurface case is a result due to Saito \cite{Saito-MLCT}). One of the main results in \cite{CDMO} 
gave a description in terms of the minimal exponent of a hypersurface, as follows. Suppose  that $Z$ is defined in $X$ by $f_1,\ldots,f_r\in\cO_X(X)$ and 
$g=\sum_{j=1}^rf_jy_j\in \cO_Y(Y)$, where $Y=X\times {\mathbf A}^r$, with $y_1,\ldots,y_r$ being the coordinates on ${\mathbf A}^r$. If $W=X\times \big({\mathbf A}^r\smallsetminus\{0\}\big)$,
then $\widetilde{\alpha}(Z)=\widetilde{\alpha}(g\vert_W)$. This description allows deducing the main properties of the minimal exponent of local complete intersections from the corresponding properties of the invariant in the case of hypersurfaces. Results on the $V$-filtration from \cite{BMS} allowed us to relate again the minimal exponent to the log canonical threshold and to rational singularities:
we have
$${\rm lct}(X,Z)=\min\big\{\widetilde{\alpha}(Z),r\big\}$$
and $\widetilde{\alpha}(Z)>r$ if and only if $Z$ has rational singularities. 
It was also shown in \cite{CDMO} that one can use the minimal exponent to detect how far the Hodge filtration on the local cohomology
$\cH^r_Z(\cO_X)$ agrees with the pole order filtration, extending the corresponding result for hypersurfaces from \cite{Saito-MLCT} and \cite{MP}. In conjunction with results from \cite{MP3}, this implied that the minimal exponent detects the
higher Du Bois property of local complete intersections. The fact that it also detects higher rational singularities in this setting was subsequently shown in \cite{CDM}. Finally, the minimal exponent can be described
in terms of the Bernstein-Sato polynomial $b_{\mathbf f}(s)$, associated to ${\mathbf f}=(f_1,\ldots,f_r)$, that was introduced in \cite{BMS}: in this case we have $b_{\mathbf f}(-r)=0$ and it was shown in \cite{Dirks} that 
$\widetilde{\alpha}(Z)$ is the negative of the largest root of 
$b_{\mathbf f}(s)/(s+r)$. 

While many of the basic properties of the minimal exponent are by now understood in the local complete intersection case, there are few known explicit examples beyond
codimension~$1$. One example given in \cite{CDMO} is that of a complete intersection in ${\mathbf A}^n$, with an isolated singularity at $0$, defined by homogeneous equations of the same degree $d$;
in this case we have $\widetilde{\alpha}(Z)=\tfrac{n}{d}$, extending a well-known formula for hypersurfaces. Our main result in this note is the following extension to the case when the homogeneous 
equations defining $Z$ have possibly different degrees:

\begin{thm}\label{thm_example}
Let $f_1,\ldots,f_r\in {\mathbf C}[x_1,\ldots,x_n]$ be homogeneous polyomials that form a regular sequence, with ${\rm deg}(f_i)=d_i$ for $1\leq i\leq r$, and such that $2\leq d_1\leq\ldots\leq d_r$.
For every $i$, we denote by $H_i$ the hypersurface defined by $f_i$ in ${\mathbf A}^n$ and by $Z$ the intersection $H_1\cap\ldots\cap H_r$. 
If on ${\mathbf A}^n\smallsetminus \{0\}$ each $H_i$ is smooth and $\sum_{i=1}^rH_i$ has simple normal crossings, 
 then
 \begin{equation}\label{eq_thm_example}
\widetilde{\alpha}(Z) = \min\big\{i+\tfrac{1}{d_i}(n-d_1-\ldots-d_i)\mid 1\leq i\leq r\big\}=p+\tfrac{1}{d_p}(n-d_1-\ldots-d_p),
\end{equation}
where $p$ is the smallest $i\leq r$ that satisfies $d_1+\ldots+d_i>n$ (with the convention that $p=r$ if there is no such $i$).
\end{thm}

We are interested, in particular, in the case when $\widetilde{\alpha}(Z)>{\rm lct}(X,Z)$, that is, when $\widetilde{\alpha}(Z)>r$. The formula in the theorem implies that this is the case if and only if $\sum_{i=1}^rd_i<n$.
We also recover the well-known facts that under the assumptions in the theorem, the pair $(X,rZ)$ is log canonical if and only if $\sum_{i=1}^rd_i\leq n$ and $Z$ has
rational singularities if and only if $\sum_{i=1}^rd_i<n$.

We also note that if $d_1=\ldots=d_r$ and we only assume that $Z\cap \big({\mathbf A}^n\smallsetminus\{0\}\big)$ is smooth, then after replacing each $f_i$ by a general linear combination of $f_1,\ldots,f_r$,
the Kleinman-Bertini theorem implies the condition that on ${\mathbf A}^n\smallsetminus\{0\}$ each $H_i$ is smooth and $\sum_{i=1}^rH_i$ has simple normal crossings. 
Therefore the above theorem implies the formula for the minimal exponent in \cite[Example~4.23]{CDMO}.

The upper bound for $\widetilde{\alpha}(Z)$ in Theorem~\ref{thm_example} can be extended to the weighted homogeneous case, even without assuming that the equations themselves are homogeneous.
The hypersurface case follows directly from a well-known formula for the minimal exponent of an isolated singularity that is nondegenerate with respect to its Newton polyhedron and the semicontinuity
of the minimal exponent in families. We then obtain the following result for complete intersections: consider on 
$R={\mathbf C}[x_1,\ldots,x_n]$ the grading such that ${\rm deg}(x_i)=w_i>0$ for $1\leq i\leq n$. For every nonzero $f\in R$, we denote by ${\rm wt}(f)$ the smallest degree of a monomial
$x^u=x_1^{u_1}\cdots x_n^{u_n}$ that 
appears with a nonzero coefficient in $f$. 

\begin{thm}\label{thm_upper_bound}
With the above notation, suppose that $f_1,\ldots,f_r\in (x_1,\ldots,x_n)^2\subseteq R$ are such that ${\rm wt}(f_i)=d_i$, for $1\leq i\leq r$, with
$d_1\leq d_2\leq\ldots\leq d_r$. If $Z$ is 
a complete intersection of pure codimension $r$ in some neighborhood of $0$,
then
$$\widetilde{\alpha}_0(Z)\leq \min\big\{i+\tfrac{1}{d_i}(w_1+\ldots+w_n-d_1-\ldots-d_i)\mid 1\leq i\leq r\big\}.$$
\end{thm}

For the precise definition of $\widetilde{\alpha}_0(Z)$ the local version of the minimal exponent of $Z$, see Section~\ref{Section2}.
We expect that if, in addition, $f_1,\ldots,f_r$ are homogeneous with respect to the above grading and the hypersurfaces $H_i$ defined by $f_i$ satisfy 
a suitable transversality assumption on ${\mathbf A}^n\smallsetminus\{0\}$ (for example, each $H_i$ is irreducible and 
$\sum_{i=1}^rH_i$ has simple normal crossings in ${\mathbf A}^n\smallsetminus 0\}$),
then the inequality in Theorem~\ref{thm_upper_bound} is an equality. When all $f_i$ have the same degree, this can be proved as in \cite[Example~4.23]{CDMO}.
When the degrees are different, however, we can only prove the assertion in the usual homogeneous case. 

The key ingredient in the proof of the lower bound for $\widetilde{\alpha}(Z)$ in Theorem~\ref{thm_example}
is a result of independent interest, giving a lower bound for the minimal exponent of a local complete intersection $Z$ in $X$ in terms of a suitable resolution of $(X,Z)$: a 
\emph{strong factorizing resolution} in the sense of Bravo and Villamayor \cite{BV}. Under the assumption that $Z$ is generically reduced, this is a proper morphism $\pi\colon \widetilde{X}\to X$ which is an isomorphism over the complement
$X\smallsetminus Z_{\rm sing}$ of the singular locus of $Z$,
with $\widetilde{X}$ smooth, and such that the reduced exceptional divisor $E$ and the strict transform $\widetilde{Z}$ of $Z$ have simple normal crossings, and $\widetilde{Z}$
is smooth. Moreover, we have a factorization
\begin{equation}\label{eq_factorization}
\cI_Z\cdot\cO_{\widetilde{X}}=\cI_{\widetilde{Z}}\cdot\cO_{\widetilde{X}}(-F),
\end{equation}
for an effective divisor $F$ supported on $E$, where $\cI_Z$ and $\cI_{\widetilde{Z}}$ are the ideals of $Z$ and $\widetilde{Z}$ in $X$ and $\widetilde{X}$, respectively. Note that the usual Hironaka algorithm does not guarantee the latter condition; the existence of strong factorizing resolutions for all generically reduced $Z$ is the main result of \cite{BV}. Given such a resolution, we write $E=\sum_{j=1}^NE_j$ as the sum of prime divisors and for every $j$,
we denote by $a_j$ and $k_j$ the coefficients of $E_j$ in the divisors $F$ and, respectively, the relative canonical divisor $K_{\widetilde{X}/X}$.

\begin{thm}\label{thm_lower_bound}
Suppose that $X$ is a smooth complex algebraic variety and $Z$ is a reduced subscheme of $X$ that is a local complete intersection, of pure codimension $r$. If $\pi\colon \widetilde{X}\to X$ is a strong factorizing resolution
of $(X,Z)$ as above, then
$$\widetilde{\alpha}(Z)\geq\min_{1\leq j\leq N}\frac{k_j+1}{a_j}.$$
\end{thm}

Note that if $Z$ is a hypersurface in $X$, then the condition (\ref{eq_factorization}) is automatically satisfied, hence a strong factorizing resolution is simply a log resolution of $(X,Z)$
such that $\widetilde{Z}$ is smooth.
In this case, the inequality in Theorem~\ref{thm_lower_bound} was proved in \cite[Corollary~D]{MP} using the theory of Hodge ideals (see also \cite[Corollary~1.5]{DM} for a more elementary proof). 
We deduce the general case in Theorem~\ref{thm_lower_bound} by reducing it to the case of hypersurfaces. 
In order to get the lower bound for $\widetilde{\alpha}(Z)$ in Theorem~\ref{thm_example}, we construct an explicit strong factorizing resolution of $({\mathbf A}^n,Z)$.

\section{An upper-bound in the weighted homogeneous case}\label{Section2}

Our goal in this section is to prove Theorem~\ref{thm_upper_bound}. 
Let us begin by recalling the local version of the minimal exponent discussed in the Introduction. If $Z$ is a local complete intersection in the smooth variety $X$,
of pure codimension $r$, and $P\in Z$, then 
for every open neighborhood $U$ of $P$, we have  $\widetilde{\alpha}(Z\cap U)\geq\widetilde{\alpha}(Z)$ and $\widetilde{\alpha}(Z\cap U)$ is constant if $U$ is small enough.
This constant value is denoted by $\widetilde{\alpha}_P(Z)$. It is then easy to see that $\widetilde{\alpha}(Z)=\min_{P\in Z}\widetilde{\alpha}_P(Z)$. We refer to  
\cite[Definition~4.16]{CDMO} and the discussion around it for details. Of course, $\widetilde{\alpha}_P(Z)$ is defined if we only know that $Z$ is a local complete intersection of codimension
$r$ at $P$. If $Z$ is a hypersurface defined by $f$, we also write $\widetilde{\alpha}_P(f)$ for $\widetilde{\alpha}_P(Z)$.

We begin with the following result in the case of hypersurfaces. We let $R={\mathbf C}[x_1,\ldots,x_n]$ and use the notation in Theorem~\ref{thm_upper_bound}.

\begin{prop}\label{prop_bound1}
If $f\in R$ is nonzero and $0\in Z$ is a singular point, then
$$\widetilde{\alpha}_0(f)\leq \frac{w_1+\ldots+w_n}{{\rm wt}(f)}.$$
\end{prop} 

\begin{proof}
We write $f=\sum_{u\in\Lambda}a_ux^u$, with $\Lambda$ finite and 
$a_u\neq 0$ for all $u\in \Lambda$.  Let $N\geq 2$ be such that $Nw_i>{\rm wt}(f)$ for all $i$. 
We consider the family of hypersurfaces parametrized by the open subset $U\subseteq \A^{|\Lambda|+n}$, with the hypersurface corresponding to 
$v=\big((c_u)_{u\in\Lambda},b_1,\ldots,b_n\big)$ being defined by $h_v=\sum_{u\in\Lambda}c_ux^u+b_1x_1^N+\ldots+b_nx_n^N$ (here $U$ consists of those 
$v$ such that $h_v$ is nonzero).
It is clear that for  $v\in U$ general, $h_v$ has an isolated singularity at $0$ and it is nondegenerate with respect to its Newton polyhedron $P$
(recall that $P$ is the convex hull of $\bigcup_{u}(u+{\mathbf R}_{\geq 0}^n)$, where the union over all monomials $x^u$ that appear with nonzero
coefficient in the equation $h$ of the hypersurface).
In this case, it is known that the minimal exponent at $0$ of such a hypersurface is $1/c$, where
$$c=\min\big\{t>0\mid (t,\ldots,t)\in P\big\}$$
(see \cite{Varchenko}, \cite{EhlersLo}, or \cite{Saito-exponents}).
Note that $P$ is the convex hull of $\big(\Lambda\cup\{Ne_1,\ldots,Ne_n\}\big)+{\mathbf R}^n_{\geq 0}$, where $e_1,\ldots,e_n$ is the standard basis of ${\mathbf Z}^n$. Since $\sum_{i=1}^nu_iw_i\geq {\rm wt}(f)$ for all $u\in\Lambda\cup\{Ne_1,\ldots,Ne_n\}$, it follows that
$\sum_{i=1}^nu_iw_i\geq {\rm wt}(f)$ for all $u\in P$, and thus $c\cdot\sum_{i=1}^nw_i\geq~{\rm wt}(f)$. 
On the other hand, it follows from the semicontinuity of minimal exponents (see \cite[Theorem~E(2)]{MP}) that for every $v'\in U$,
we have $\widetilde{\alpha}_0(h_{v'})\leq \widetilde{\alpha}_0(h_v)=1/c$, when $v\in U$ general.
In particular, this applies for $f$, and we get
\[\widetilde{\alpha}_0(f)\leq \frac{1}{c}\leq \frac{w_1+\ldots+w_n}{{\rm wt}(f)}.\]
\end{proof}

Before giving the proof of Theorem~\ref{thm_upper_bound}, we give a lemma that describes the infimum in this theorem.

\begin{lem}\label{lem-eq}
Let $w\in {\mathbf R}$ and 
let $d_1\leq\ldots\leq d_r$ be positive integers. If for $1\leq i\leq r$, we put
$$
\alpha_i:=i+\tfrac{1}{d_i}(w-d_1-\ldots-d_i),
$$
then the following hold:
\begin{enumerate}
\item[i)] If $i\leq r-1$ is such that $d_i=d_{i+1}$, then $\alpha_i=\alpha_{i+1}$.
\item[ii)] If $i\leq r-1$ and $d_i<d_{i+1}$, then $\alpha_i\geq\alpha_{i+1}$ if and only if $d_1+\ldots+d_i\leq w$.
\item[ii)] We have $\min_i\alpha_i=\alpha_p$, 
where $p$ is the smallest $i\leq r$ that satisfies $d_1+\ldots+d_i>w$ (with the convention that $p=r$ if there is no such $i$).
\end{enumerate}
\end{lem}

\begin{proof}
The first two assertions follow from the fact that for $i\leq r-1$, we have
$$\alpha_i-\alpha_{i+1}=\frac{(w-d_1-\ldots-d_i)(d_{i+1}-d_i)}{d_id_{i+1}},$$
and the third assertion is an easy consequence.
\end{proof}

We can now prove the upper bound for the minimal exponent of complete intersections in terms of the weights of the defining equations.

\begin{proof}[Proof of Theorem~\ref{thm_upper_bound}]
For every $i$, with $1\leq i\leq r$, let
$$\alpha_i=i+\tfrac{1}{d_i}(w_1+\ldots+w_n-d_1-\ldots-d_i),
$$
and let $p$ be such that $\alpha_p=\min_i\alpha_i$. By Lemma~\ref{lem-eq}i), we may assume that if $p>1$, then $d_{p-1}<d_p$. 

Let $g=\sum_{j=1}^rf_jy_j\in \cO({\mathbf A}^n\times {\mathbf A}^r)$, where $y_1,\ldots,y_r$ are the coordinates on ${\mathbf A}^r$.
It follows from the description of the minimal exponent of $Z$ in terms of $g$ given in the Introduction that
if $U={\mathbf A}^r\smallsetminus \{0\}\supseteq U'=(y_p\neq 0)$, then
$$\widetilde{\alpha}_0(Z)=\max_{V\ni 0}\widetilde{\alpha}(g\vert_{V\times U})\leq \max_{V\ni 0}\widetilde{\alpha}(g\vert_{V\times U'}),$$
where $V$ runs over the open neighborhoods of $0$ in ${\mathbf A}^n$. We put $z_j=y_j/y_p$ for $1\leq j\leq r$, $j\neq p$,
so $z_1,\ldots,\widehat{z_p},\ldots,z_r$ can be viewed as coordinates
on ${\mathbf A}^{r-1}$. 
Since $g$ is homogeneous 
of degree $1$ with respect to $y_1,\ldots,y_r$, it follows that if we put 
$$h=g/y_p=f_1z_1+\ldots+f_{p-1}z_{p-1}+f_p+f_{p+1}z_{p+1}+\ldots+f_rz_r\in\cO({\mathbf A}^n\times {\mathbf A}^{r-1}),$$
then 
$$\widetilde{\alpha}(g\vert_{V\times U'})=
\widetilde{\alpha}(h\vert_{V\times {\mathbf A}^{r-1}})$$
(we use here the fact that the minimal exponent does not change by pull-back by a smooth surjective morphism, see for example \cite[Proposition~4.12]{CDMO}).
We thus conclude that
\begin{equation}\label{eq1_upper_bound}
\widetilde{\alpha}_0(Z)\leq \widetilde{\alpha}_{(0,0)}(h).
\end{equation}

By assumption, we have $f_p\in (x_1,\ldots,x_n)^2$,  and thus $h$ has a singular point at $(0,0)$. 
If we consider the weight of $z_j$ to be $d_p-d_j$ for $1\leq j\leq p-1$ and $\epsilon>0$ for $p+1\leq j\leq r$,
then we see that ${\rm wt}(h)=d_p$, hence it follows from Proposition~\ref{prop_bound1} that
\begin{equation}\label{eq2_upper_bound}
\widetilde{\alpha}_{(0,0)}(h)\leq \frac{1}{d_p}\big(w_1+\ldots+w_n+(d_p-d_1)+\ldots+(d_p-d_{p-1})+(r-p)\epsilon\big)
=\alpha_p+\frac{r-p}{d_p}\epsilon.
\end{equation}
By combining (\ref{eq1_upper_bound}) and (\ref{eq2_upper_bound}), and letting $\epsilon$ go to $0$, we obtain the inequality
in the theorem.
\end{proof}

\section{A general lower bound via a strong factorizing resolution}

In this section, we prove the lower bound on the minimal exponent in terms of a strong factorizing resolution.

\begin{proof}[Proof of Theorem~\ref{thm_lower_bound}]
We may and will assume that $X$ is affine and $Z$ is defined by a regular sequence $f_1,\ldots,f_r\in\cO_X(X)$. Let $g=f_1y_1+\ldots+f_ry_r\in\cO_Y(Y)$, where
$Y=X\times \A^r$, with $y_1,\ldots, y_r$ being the coordinates on $\A^r$. Let $W=X\times \big(\A^r\smallsetminus\{0\}\big)$, so
$\widetilde{\alpha}(Z)=\widetilde{\alpha}(g\vert_W)$.

 Consider now the morphism 
$$\varphi=\pi\times{\rm id}_{\A^r}\colon \widetilde{Y}=\widetilde{X}\times\A^r\to Y.$$
This is a projective morphism which is an isomorphism over the complement of $Z_{\rm sing}\times\A^r$. The exceptional divisors of $\varphi$ are the
$E_i\times\A^r$, with $1\leq i\leq N$. Moreover, it follows from the definition of a strong factorizing resolution that we can cover
$\widetilde{X}$ by open subsets $V_j$, such that on each $V_j\times \A^r$ we can write 
$$g\circ\varphi\vert_{V_j\times\A^r}=v_j\cdot\sum_{i=1}^rh_iy_i,$$
where the divisor ${\rm div}(v_j)$ defined by $v_j$ is supported on $G=E\times\A^r$ and $h_1,\ldots,h_r$ generate the ideal of $\widetilde{Z}$ in $V_j$. 
Moreover, the coefficient of $E_i\times {\mathbf A}^r$ in ${\rm div}(v_j)$ is $a_i$. Note that if $V_j\cap\widetilde{Z}=\emptyset$, then $\sum_{i=1}^rh_iy_i$ defines a smooth
hypersurface in $V_j\times {\mathbf A}^r$, that has simple normal crossings with $G$.

By assumption, $\widetilde{Z}$ is smooth, of codimension $r$ in $\widetilde{X}$, and has simple normal crossings with $E$ (that is, both $E$ and $E\vert_{\widetilde{Z}}$ are reduced simple normal crossing
divisors). Therefore we may and will assume that for every $j$ such that $V_j\cap\widetilde{Z}\neq\emptyset$,
we have algebraic coordinates $x_1,\ldots,x_n$ on $V_j$ such that $h_i=x_i$ for $i\leq r$
and $E\vert_{V_j}=\sum_{i=r+1}^{r+s}a_i\cdot {\rm div}(x_i)$.

Let $\varphi_W\colon \varphi^{-1}(W)\to W$ be the restriction of $\varphi$ over $W$. Note that on $\varphi^{-1}(W)\cap (V_j\times \A^r)$, with $V_j\cap\widetilde{Z}\neq\emptyset$, the divisor defined by
$$(x_1y_1+\ldots+x_ry_r)\cdot\prod_{i=r+1}^{r+s}x_i^{a_i}$$
has simple normal crossings. Since $g\circ\varphi$ clearly defines a simple normal crossing divisor in $\varphi^{-1}(W)\cap (V_j\times \A^r)$ when $V_j\cap\widetilde{Z}=\emptyset$,
we conclude that $\varphi_W$ is a log resolution of $(W, {\rm div}(g)\vert_W)$ which is an isomorphism over $W\smallsetminus V(g)$. 
The exceptional divisors of $\varphi_W$ are the $E'_i=E_i\times \big(\A^r\smallsetminus \{0\}\big)$ and the relative canonical divisor of $\varphi_W$ is
$\sum_{i=1}^Nk_iE'_i$. Moreover, the divisor ${\rm div}(g)\vert_W$ is reduced: its singular locus is contained in $Z_{\rm sing}\times \big({\mathbf A}^r\smallsetminus\{0\}\big)$
(see \cite[Lemma~4.22]{CDMO}) and thus ${\rm div}(g)$ is generically reduced, hence reduced. In addition, its strict transform on $\varphi^{-1}(W)$ is smooth:
this is clear on $V_j\times \big({\mathbf A}^r\smallsetminus \{0\}\big)$ if $V_j\cap\widetilde{Z}=\emptyset$, while if $V_j\cap\widetilde{Z}\neq\emptyset$, it follows from the fact that 
it is defined by $\sum_{i=1}^rx_iy_i$.
We can thus apply the lower bound on the minimal exponent of a hypersurface
in terms of a log resolution (see \cite[Corollary~D]{MP} or \cite[Corollary~1.5]{DM})
to conclude that 
$$\widetilde{\alpha}(Z)=\widetilde{\alpha}(g\vert_U)\geq \min_{1\leq i\leq N}\frac{k_i+1}{a_i},$$
which is the assertion in the theorem.
\end{proof}

\section{The formula in the homogeneous case}

Our main goal in this section is to prove Theorem~\ref{thm_example}. 
In order to prove 
the lower bound in the theorem, we will use Theorem~\ref{thm_lower_bound}. We thus proceed to describe 
a strong factorizing resolution of $({\mathbf A}^n, Z)$. 

With the notation in Theorem~\ref{thm_example}, let $\pi_1\colon X_1\to {\mathbf A}^n$
be the blow-up of the origin, with exceptional divisor $E_1$. Suppose that $k\geq 1$ and $1\leq p_1, p_2,\ldots,p_k$ are such that
$$d_1=\ldots=d_{p_1}<d_{p_1+1}=\ldots=d_{p_1+p_2}<\ldots <d_{p_1+\ldots+p_{k-1}+1}=\ldots=d_{p_1+\ldots+p_k}.$$
Note that $p_1+\ldots+p_k=r$.
In order to simplify the notation, we put $e_j=d_{p_1+\ldots+p_j}$ for $1\leq j\leq k$.
We define a morphism $\pi\colon Y\to {\mathbf A}^n$ to be the composition of $\pi_1$ with $\sum_{i=1}^{k-1}(e_{i+1}-e_i)$ smooth blow-ups, as follows.
First, we consider $(e_2-e_1)$ blow-ups, each of these blowing up the intersection of the previous exceptional divisor with the strict transforms of
$H_1,\ldots,H_{p_1}$. We next consider $(e_3-e_2)$ blow-ups, each of these blowing up the intersection of the previous exceptional divisor with the strict transforms of
$H_1,\ldots,H_{p_1+p_2}$, etc. 

\begin{prop}\label{strong_res}
With the above notation, the composition $\pi\colon Y\to {\mathbf A}^n$ has the following properties:
\begin{enumerate}
\item[i)] If $r\leq n-1$, then $\pi$ is a strong factorizing resolution of $({\mathbf A}^n,Z)$.
\item[ii)] If $r=n$, then $\pi$ is a log resolution of the pair $({\mathbf A}^n,Z)$. 
\end{enumerate}
\end{prop}

\begin{proof}
We note that if $r\leq n-1$, then the assumption on $Z$ implies that it is generically reduced, hence reduced, since it is a complete intersection and thus Cohen-Macaulay. 
Therefore, in this case, it makes sense to say that $\pi$ is a strong factorizing resolution. 

The blow-up $X_1$ is covered by affine open charts $U_1,\ldots,U_n$, where $U_i$ has coordinates 
$$x_i, y_1,\ldots,y_{i-1},y_{i+1},\ldots,y_n$$ such that
$x_j=x_iy_j$ for all $j\neq i$. Note that if $I_Z=(f_1,\ldots,f_r)$, then 
$$I_Z\cdot \cO_{U_i}=(x_i^{d_1}g_1,\ldots,x_i^{d_r}g_r),$$
where $g_j=f_j(y_1,\ldots,y_{i-1},1,y_{i+1},\ldots,y_n)$ for $1\leq j\leq r$. Moreover, $E_1\cap U_i$ is defined by $x_i$ and if $r<n$, then the strict transform $\widetilde{Z}$ of $Z$ 
on $X_1$ is defined in $U_i$ by $(g_1,\ldots,g_r)$, hence it is smooth.
Note that if $k=1$ (that is, we have $d_1=\ldots=d_r$), then $\pi_1$ is a strong factorizing resolution when $r<n$ and is a log resolution of $(X,Z)$ for $r=n$.
Therefore we are done in this case.

Suppose now that $k>1$.
We note that, by our assumption on $f_1,\ldots,f_r$, the hypersurfaces defined by $x_i,g_1,\ldots,g_r$ in $U_i$ are smooth and their sum has simple normal crossings. 
It follows that for every point $P\in U_i$, we can find algebraic coordinates $z_0,\ldots,z_{n-1}$ in a neighborhood $W_P$ of $P$ such that the ideal $I_Z\cdot \cO_{W_P}$
is equal to: 

\noindent {\bf Case 1}. $(z_1,\ldots,z_r)$. This is the case when $P\not\in E_1$, when we may assume that $W_P\cap E_1=\emptyset$. This case is clear: it follows from 
the definition of $\pi$ that the morphism $Y\to X_1$ is an isomorphism over $W_P$ and it is clear that 
above $W_P$ the condition for $\pi$ to be a strong factorizing resolution (if $r<n$) or a 
log resolution of $(X,Z)$ (if $r=n$) is satisfied. 

\noindent {\bf Case 2}. $(z_0^{d_1}z_1,\ldots,z_0^{d_r}z_{r})$. This is the case when $r<n$ and $P$ lies on the hypersurfaces defined by $x_i,g_1,\ldots,g_r$.
Note that $\widetilde{Z}$ is defined in $W_P$ by $(z_1,\ldots,z_r)$.

\noindent {\bf Case 3}. $(z_0^{d_1}z_1,\ldots,z_0^{d_q}z_q, z_0^{d_{q+1}})$, for some $q<r$. This is the case when $P$ lies on $E_1$ and $g_j(P)=0$ for $j\leq q$, but
$g_{q+1}(P)\neq 0$. After getting rid of some redundant generators, we may assume that $q=p_1+\ldots+p_m$. If $r<n$, then we see that $\widetilde{Z}$ does not meet
$W_P$ in this case. 

We now consider the next blow-up $\pi_2\colon X_2\to X_1$ in our sequence: we blow up along $E_1\cap \widetilde{H_1}\cap \ldots\cap\widetilde{H_{p_1}}$,
where $\widetilde{H_j}$ denotes the strict transform of $H_j$ on $X_1$. Let's describe $\pi_2$ over the above open subset $W_P\subseteq U_i$ when we are in Case 2 or Case 3.
Note that we are blowing up along the zero locus
of $(z_0,z_1,\ldots,z_{p_1})$, which is smooth.
Let $V_j$ be the chart in $\pi_2^{-1}(W_P)$ given by 
$$z_{\ell}=u_{\ell}\quad\text{for}\quad\ell\in \{j,p_1+1,\ldots,n\} \quad\text{and}\quad z_{\ell}=u_ju_{\ell}\quad\text{for}\quad 0\leq\ell\leq p_1, \ell\neq j$$
for some $j$, with $1\leq j\leq p_1$.
An easy computation shows that $I_Z\cdot\cO_{V_j}$ is generated by $u_j^{d_1+1}u_0^{d_1}$ in both Cases 2 and 3. Since $u_j$ defines the $\pi_2$-exceptional divisor
and $u_0$ defines the strict transform of $E_1$, we see that $I_Z\cdot\cO_{V_j}$ is the ideal of a divisor supported on the exceptional locus. Therefore the condition
for a strong factorizing resolution (in the case $r<n$) or for a log resolution (in the case $r=n$) will be trivially satisfied over $V_j$. 

We next consider the chart $V_0$ in  $\pi_2^{-1}(W_P)$ given by 
$$z_{\ell}=u_{\ell}\quad\text{for}\quad\ell\in \{0,p_1+1,\ldots,n\} \quad\text{and}\quad z_{\ell}=u_0u_{\ell}\quad\text{for}\quad 1\leq\ell\leq p_1.$$
Note that the $\pi_2$-exceptional divisor is defined in this chart by $u_0$. 
Again, an easy computation shows that $I_Z\cdot\cO_{V_0}$ is equal to
$$(u_0^{e_1+1}u_1,\ldots u_0^{e_1+1}u_{p_1},u_0^{e_2}u_{p_1+1},\ldots,u_0^{e_k}u_r)$$
in Case 2 and to
$$(u_0^{e_1+1}u_1,\ldots,u_0^{e_1+1}u_{p_1},\ldots, u_0^{e_{m+1}})$$
in Case 3 (we recall that $m$ is such that $q=p_1+\ldots+p_m$).
We thus see that if we are in Case 2, after performing $(e_2-e_1)$ such blow-ups, we are in the situation where $k$ is replaced by $k-1$:
in the only charts that we need to consider, we have coordinates $v_0,v_1,\ldots,v_{n-1}$, such that the pull-back of $I_Z$ is equal to
$$(v_0^{e_2}v_1,\ldots,v_0^{e_2}v_{p_1+p_2}, v_0^{e_3}v_{p_1+p_2+1},\ldots,v_0^{e_k}v_r).$$
If $k>2$, then the next blow-up is along the ideal $(v_0,v_1,\ldots,v_{p_1+p_2})$, and the process continues as above. 
In the end, we see that in the only charts that we need to consider, 
we have coordinates $w_0,\ldots,w_{n-1}$ such that the pull-back of $I_Z$ is $w_0^{d_r}\cdot (w_1,\ldots,w_r)$. 
Therefore, in such a chart, the condition for having a strong factorizing resolution is satisfied.

Similarly, if we are in Case 3, then after the first $(e_2-e_1)+\ldots+(e_{m+1}-e_{m})$ blow-ups, in the only charts that we need to consider,
we have coordinates $w_0,\ldots,w_{n-1}$ such that the pull-back of $I_Z$ is $(w_0^{e_{m+1}})$. Therefore, in this chart, we only have the ideal of a divisor 
supported on the exceptional locus, so this satisfies the condition for $\pi$ to be a strong factorizing resolution when $r<n$ and to be a log resolution when $r=n$.
This completes the proof of the proposition.
\end{proof}

\begin{rmk}\label{rmk_vanish_b_r}
With the notation in Proposition~\ref{strong_res}, it follows from the definition of $\pi$ that if $r<n$, then starting with $X_1$, at each step we blow up a smooth center that is not contained in
the strict transform of $Z$ on the respective variety. In fact, with the notation in the proof, for every chart $U_i$ on $X_1$ and for every exceptional divisor on $Y$
whose image in $X_1$ intersects $U_i$, we see that $g_r$ does not vanish along this image. 
\end{rmk}

We can now prove the main result of this note.

\begin{proof}[Proof of Theorem~\ref{thm_example}]
For every $k$, with $1\leq k\leq n$, we put
$\alpha_k=k+\frac{n-d_1-\ldots-d_k}{d_k}$.
By Lemma~\ref{lem-eq}, we know that $\min_k\alpha_k=\alpha_p$, where $p$ is as in the statement of the theorem.
Since the inequality
$$\widetilde{\alpha}(Z) \leq \alpha_p$$
follows from Theorem~\ref{thm_upper_bound}, we only need to prove the opposite inequality.

Suppose first that $\sum_{j=1}^rd_j>n$. Note that since $\widetilde{\alpha}(Z)\leq\alpha_n<r$, we know that in this case we have
$\widetilde{\alpha}(Z)={\rm lct}(X,Z)$, and thus only need to show that ${\rm lct}(X,Z)\geq\alpha_p$. For basic facts about log canonical thresholds
(including the definition), we refer to \cite[Chapter~9]{Lazarsfeld}.
As in the proof of Proposition~\ref{strong_res}, we consider the blow-up $\pi_1\colon X_1\to {\mathbf A}^n$ of ${\mathbf A}^n$, with exceptional divisor $E_1$.
Note that $K_{X_1/{\mathbf A}^n}=(n-1)E_1$. We have seen in the proof of Proposition~\ref{strong_res} that we can cover $X_1$ by affine open charts $U_i$,
such that $I_Z\cdot\cO_{U_i}=(x_i^{d_1}g_1,\ldots,x_i^{d_r}g_r)$, where $x_i$ defines $E_1$ in $U_i$, and the divisors defined by $x_i,g_1,\ldots,g_r$ are smooth
and their sum has simple normal crossings. We need to show that if $G$ is a prime divisor on $W$, where $\varphi\colon W\to X_1$ is such that $\pi_1\circ\varphi$ is a 
log resolution of $(X,Z)$,
 with the valuation ${\rm ord}_G$ corresponding to $G$, and if $a_G={\rm ord}_G(I_Z)$
and $k_G$ is the coefficient of $G$ in $K_{W/X}$, then $\tfrac{k_G+1}{a_G}\geq\alpha_p$. 
Suppose that the image of $G$ on $X_1$ intersects the chart $U_i$ and let $b_0={\rm ord}_G(x_i)$ and $b_j={\rm ord}_G(g_j)$ for $1\leq j\leq r$. 
We may and will assume that $b_0>0$: otherwise, since $Z\smallsetminus\{0\}$ is smooth, of codimension $r$ in ${\mathbf A}^n\smallsetminus\{0\}$,
we have ${\rm lct}({\mathbf A}^n\smallsetminus\{0\},Z\smallsetminus\{0\})=r$, and thus $\tfrac{k_G+1}{a_G}\geq r>\alpha_p$.
It is well-known that since the divisor ${\rm div}(x_i)+\sum_{j=1}^r{\rm div}(g_j)$ has simple normal crossings, 
if $k'_G$ is the coefficient of $G$ in $K_{W/X_1}$, then
$$k'_G+1\geq b_0+b_1+\ldots+b_r$$
(see, for example, the proof of \cite[Lemma~9.2.19]{Lazarsfeld}). 
Since 
$$K_{W/X}=K_{W/X_1}+\varphi^*(K_{X_1/{\mathbf A}^n})=K_{W/X_1}+(n-1)\varphi^*(E_1),$$
 we have
$$k_G+1=k'_G+1+(n-1)b_0\geq nb_0+\sum_{j=1}^rb_j.$$
Since 
$${\rm ord}_G(I_Z)=\min\{b_0d_j+b_j\mid 1\leq j\leq r\},$$
it follows that it is enough to show that
\begin{equation}\label{ineq1}
nb_0+\sum_{j=1}^rb_j\geq\alpha_p\cdot\min\{b_0d_j+b_j\mid 1\leq j\leq r\}.
\end{equation}
If we put $u_j=b_j/b_0$ for $1\leq j\leq r$ and $M=\min\{d_j+u_j\mid 1\leq j\leq r\}$, then
(\ref{ineq1}) becomes
\begin{equation}\label{ineq2}
n+\sum_{j=1}^ru_j\geq \alpha_pM.
\end{equation}
We define an increasing sequence $k_1<k_2<\ldots<k_s=r$ such that
$$k_1=\max\{j\mid 1\leq j\leq r, d_j+u_j=M\},$$
and if $k_{\ell}<r$, then
$$k_{\ell+1}=\max\{k>k_{\ell}\mid d_k+b_k=\min\{d_j+u_j\mid j>k_{\ell}\}\big\}.$$
With this notation, the  inequality (\ref{ineq2}) becomes
\begin{equation}\label{ineq3}
\frac{n+u_1+\ldots+u_r}{d_{k_1}+u_{k_1}}\geq\alpha_p.
\end{equation}
 For $1\leq q\leq s$, let us put
\begin{equation}\label{ineq4}
\beta_{k_q}:=\frac{n+k_qu_{k_q}+\sum_{j=1}^{k_q}(d_{k_q}-d_j)+\sum_{j>k_q}u_j}{d_{k_q}+u_{k_q}}.
\end{equation}
For every $j<k_1$, we have $u_j\geq u_{k_1}+(d_{k_1}-d_j)$, hence the left-hand side of (\ref{ineq3}) is
$\geq \beta_{k_1}$ and thus (\ref{ineq3}) follows if we show
\begin{equation}\label{ineq6}
\beta_{k_1}\geq\alpha_p.
\end{equation}
 The key step is to show that if $q<s$, then
\begin{equation}\label{ineq5}
\beta_{k_q}\geq\min\{\alpha_{k_q}, \beta_{k_{q+1}}\}.
\end{equation}
Indeed, if we view $\beta_{k_q}$ as a function of $u_{k_q}$, since $0\leq u_{k_q}\leq u_{k_{q+1}}+(d_{k_{q+1}}-d_{k_q})$,
we see that $\beta_{k_q}$ is bounded below by the minimum taken when $u_{k_q}=0$ and when $u_{k_q}=u_{k_{q+1}}+(d_{k_{q+1}}-d_{k_q})$.
In the former case, the value is 
$$\frac{n+\sum_{j=1}^{k_q}(d_{k_q}-d_j)+\sum_{j>k_q}u_j}{d_{k_q}}\geq \alpha_{k_q},$$
while in the latter case, using the fact that $u_j\geq u_{k_{q+1}}+d_{k_{q+1}}-d_j$ for $k_q<j\leq k_{q+1}$, the value is
$$\frac{n+k_q(u_{k_{q+1}}+d_{k_{q+1}}-d_{k_q})+\sum_{j=1}^{k_q}(d_{k_q}-d_j)+\sum_{j>k_q}u_j}{u_{k_{q+1}}+d_{k_{q+1}}}\geq\beta_{k_{q+1}}.$$
We thus obtain the inequality in (\ref{ineq5}). Using the fact that $\alpha_{k_q}\geq\alpha_p$ for all $q$ gives
$$\beta_{k_1}\geq\min\{\alpha_p,\beta_{k_s}\}.$$
On the other hand, we have $k_s=r$, and thus
$$\beta_{k_s}=\frac{n+ru_r+\sum_{j=1}^r(d_r-d_j)}{d_r+u_r}.$$
As above, if we view this as a function of $u_r$, we see that it is bounded below by the minimum of its values when $u_r=0$ (which is $\alpha_r$) and the
 value of the limit when $u_r$ goes to infinity (which is $r>\alpha_r$). We thus conclude that $\beta_{k_1}\geq\alpha_p$, completing the proof of (\ref{ineq6}),
 and thus the proof of the theorem when $\sum_{j=1}^rd_j>n$. 
 
 Suppose now that $\sum_{j=1}^rd_j\leq n$. Note that since we assume $d_j\geq 2$ for all $j$, we have $r<n$. 
By Theorem~\ref{thm_lower_bound}, in order to show
 that $\widetilde{\alpha}(Z)\geq\alpha_r$, it is enough to show that if $G$ is a prime $\pi$-exceptional divisor on $Y$, then we have $\tfrac{k_G+1}{a_G}\geq\alpha_r$
 (note that we keep the notation in the first part of the proof). We choose again a chart $U_i$ on $X_1$ that intersects the image of $G$ and put
 $b_0={\rm ord}_G(x_i)$ and $b_j={\rm ord}_G(g_j)$ for $1\leq j\leq r$. As before, it is enough to show that
 \begin{equation}\label{ineq8}
 \tfrac{nb_0+b_1+\ldots+b_r}{a_G}\geq \alpha_r,
 \end{equation}
 where $a_G=\min\{b_0d_j+b_j\mid 1\leq j\leq r\}$.
 A key point is that, by construction, we have $b_r=0$ (see Remark~\ref{rmk_vanish_b_r}). This implies that 
 \begin{equation}\label{ineq9}
 a_G\leq b_0d_r.
 \end{equation}
 On the other hand, since $b_j\geq a_G-b_0d_j$ for $j\geq 1$, we have
 $$ \tfrac{nb_0+b_1+\ldots+b_r}{a_G}\geq \frac{nb_0+\sum_{j=1}^r(a_G-b_0d_j)}{a_G}=r+\tfrac{(n-d_1-\ldots-d_r)b_0}{a_G}\geq \alpha_r,$$
 where the last inequality follows from (\ref{ineq9}), using the fact that $n\geq\sum_{j=1}^rd_j$. This proves (\ref{ineq8}) and completes the proof of the theorem.
\end{proof}

\begin{rmk}
In the statement of Theorem~\ref{thm_example}, we made the assumption that $d_1\geq 2$. The general case can be easily reduced to this one: indeed, suppose that 
$d_q=1<d_{q+1}$ for some $q\leq r-1$. In this case, $Z$ is isomorphic to a closed subscheme $W$ of ${\mathbf A}^{n-q}$ defined by homogeneous equations of degrees
$d_{q+1}\leq\ldots\leq d_r$, and which satisfies the hypothesis in Theorem~\ref{thm_example}. 
Moreover, by \cite[Proposition~4.14]{CDMO}, we have $\widetilde{\alpha}(Z)=\widetilde{\alpha}(W)+q$. 
\end{rmk}

\end{document}